\documentclass[11pt]{article}
\usepackage{amsmath}
\usepackage{latexsym}
\usepackage{graphicx}
\usepackage{amssymb}
\begin{document}
\makeatletter
\def\@begintheorem#1#2{\trivlist \item[\hskip \labelsep{\bf #2\ #1.}] \it}
\def\@opargbegintheorem#1#2#3{\trivlist \item[\hskip \labelsep{\bf #2\ #1}\ {\rm (#3).}]\it}
\makeatother
\newtheorem{thm}{Theorem}[section]
\newtheorem{alg}[thm]{Algorithm}
\newtheorem{conj}[thm]{Conjecture}
\newtheorem{lemma}[thm]{Lemma}
\newtheorem{defn}[thm]{Definition}
\newtheorem{cor}[thm]{Corollary}
\newtheorem{exam}[thm]{Example}
\newtheorem{prop}[thm]{Proposition}
\newenvironment{proof}{{\sc Proof}.}{\rule{3mm}{3mm}}

\title{Some graph theoretical aspects\\ of generalized truncations}
\author{Brian Alspach and Joshua B. Connor\\School of Mathematical and Physical Sciences\\
University of Newcastle\\Callaghan, NSW 2308,
Australia\\brian.alspach@newcastle.edu.au\\joshconnor178@gmail.com}
\maketitle

\begin{abstract} A broader definition of generalized truncations of graphs is introduced followed
by an exploration of some standard concepts and parameters with regard to generalized
truncations.
\end{abstract}


\medskip


\section{Introduction}

Truncations of Platonic and Archimedean solids already were studied by the ancient Greeks.
It is worth observing the use of the term ``solid'' when considering truncations.  The
act of slicing off a corner of a solid allows an immediate and intuitive understanding of what
a truncation produces.  The skeletons of these solids, that is, the graphs formed by the vertices
and edges of these solids then inherit an obvious truncation.  This suggests that a notion of
truncation may be applied to arbitrary graphs. However, some care needs to be exerted
when extending the notion of truncation to arbitrary graphs for the following reason.  Upon
truncating a vertex of a solid, the $k$ dangling edges that were incident with the vertex that has been removed
then are joined as a $k$-cycle forming the boundary of a new face in a straightforward manner.
Thus, the temptation for an arbitrary graph would be to somehow join the dangling edges so
that they form a cycle.  Indeed, this has been the case in some instances where truncation
has been employed, but other graphs have been employed as well.  We now provide a
brief discussion of some of the history in spite of delaying the precise definition of a generalized
truncation. 

H. Sachs \cite{S1} seems to be the first modern graph theorist to have used truncation to
obtain graphs with specific properties.  He did not restrict the replacement graphs to be cycles,
but did use the same graph for each replacement, and used a Hamilton cycle in each to
organize the edges between the replaceent graphs.  His work was then extended by Exoo
and Jajcay in \cite{E1}.  The gap between those two papers is essentially fifty years.

Perhaps the best known graph truncation is the cube-connected cycles graph introduced in
\cite{P1}.  It is obtained by replacing each vertex of the $n$-dimensional cube
with an $n$-cycle.  The resulting graph is trivalent and has cube-like properties.  Closely
related to this is the truncation that replaces each vertex of an arc-transitive graph with a
cycle in such a way that a trivalent vertex-transitive graph is obtained.  This is exploited
nicely in \cite{E2,E1} and elsewhere.  Another paper dealing with replacing vertices by cycles is
\cite{D1}.

When the replacement graphs are cycles, if the order of the vertices along the cycles is
not handled with some care, desirable properties of the original graph may be lost.  That
problem is addressed in \cite{A1} by using complete graphs for the replacements.
Generalized truncations also appear several times in \cite{B2}.  They are used in articles
about graph expanders under the name zig zag product (for example, see \cite{R1}).

\medskip

The purpose of this paper is the introduction of a much broader definition of generalized
truncations of graphs and an exploration of some standard graph parameters in this setting.
We believe there is considerable scope for research in this topic and include eight research
problems we encountered.  

The terms {\it reflexive} and {\it multigraph} are used if loops and multiple edges,
respectively, are allowed.  Thus, a graph has neither loops nor multiple edges.  We use $V(X)$
to denote the set of vertices of a reflexive multigraph $X$ and $E(X)$ to denote the set of
edges.  The {\it order} of $X$ is $|V(X)|$ and the {\it size} of $X$ is $|E(X)|$.  Finally,
the {\it valency} of a vertex $u$, denoted $\mathrm{val}(u)$, is the number of edges
incident with $u$, where a loop contributes 2 to the valency.

Given a reflexive multigraph $X$, a {\it generalized truncation} of $X$ is obtained as follows via a
two-step operation. The first step is the {\it excision step}.  Let $M$ denote an auxiliary 
matching (no two edges have a vertex in common) of size $|E(X)|$.  Let
$F:E(X)\rightarrow M$ be a surjective function and for $uv\in E(X)$, label the ends of the
edge $F(uv)$ with $u$ and $v$.  Let $F(M)$ denote the vertex-labelled matching thus obtained.
So $F(M)$ represents the edges of $X$ completely disassembled.  Note that a loop at a vertex
$v\in V(X)$ produces an edge in $F(M)$ with both end vertices labelled $v$.

The second step is the {\it assemblage step}.  For each $v\in V(X)$, the set of vertices of $F(M)$ labelled
with $v$ is called the {\it cluster at} $v$ and is denoted $\mathrm{cl}(v)$.  Insert an arbitrary graph on
$\mathrm{cl}(v)$.  The inserted graph on $\mathrm{cl}(v)$ is called the {\it constituent graph at $v$}
and is denoted $\mathrm{con}(v)$. The resulting multigraph $$F(M)\cup_{v\in V(X)} \mathrm{con}(v)$$
is a {\it generalized truncation} of $X$.  We usually think of the labels on the vertices of
$F(M)$ as being removed following the assemblage stage, but there are many times when the labels
are useful in the exposition.  We use $\mathrm{TR}(X)$ to denote a
generalized truncation of the reflexive multigraph $X$.

Truncations arise via action involving the edges incident with a vertex.  Consequently, isolated
vertices are useless and we make the important convention that the reflexive multigraphs from
which we are forming generalized truncations do not have isolated vertices.  This will not be
mentioned in the subsequent material, but is required for the validity of a few statements.
Note that we claim that a generalized truncation may be a multigraph.  This issue is addressed in the
next section.

\medskip

A few words about ``style'' are in order.  There are two styles we recognize:  local theorems
and global theorems.  Some discussion and two examples should clarify the distinction we are
trying to make.

A {\it local theorem} is a result that is achieved by considering only the the consituent graphs.
A {\it global theorem} is a result that requires accounting for the structure of the multigraph
$X$ in carrying out the construction producing a generalized truncation.  This description is
admittedly a little fuzzy so let's consider two examples arising later in the paper.

Theorem \ref{euler} is a local theorem even though the hypotheses require that $X$ be
eulerian.  We consider it local because once we start with an eulerian multigraph the subsequent
construction requires only that we build constituents so that every vertex has odd valency.
The structure of $X$ has nothing to do with constructing the constituents.  On the other
hand, Theorem \ref{con} is global because the choices for edges for the constituents
depends heavily on the structure of $X$.
        
\section{Some Characterizations}

According to the definition above, a generalized truncation of a reflexive multigraph $X$ may be
a multigraph.  This follows because we start with a matching and then add graphs on specified
subsets of vertices.  Thus, loops do not arise in the assemblage stage.  Multiple edges
may arise but only if $X$ has loops.  

A natural question to ask is which multigraphs are generalized truncations of a reflexive
multigraph.  One obvious fact is that a generalized truncation contains a perfect matching, but
this is not sufficient as we shall see.  Given a multigraph $X$ and a set of edges $E'\subseteq
E(X)$, we use $X\setminus E'$ to denote the submultigraph obtained from $X$ by removing
the edges in $E'$.

\begin{thm}\label{nscond}  A multigraph $Y$ is a generalized truncation of a reflexive
multigraph if and only if $Y$ contains a perfect matching $M$ such that $Y\setminus M$ is
a graph.  
\end{thm}
\begin{proof} If $Y$ is a generalized truncation of some reflexive multigraph $X$, then it
contains the edges of $F(E(X))=F(M)$ and this forms a perfect matching in $Y$.  If we remove
the edges of $F(M)$ from $Y$, the resulting submultigraph $Y\setminus F(M)$ is a graph by
definition because the constituents partition the vertex set of $Y\setminus F(M)$.

For the other direction, let $Y$ be a multigraph containing a perfect matching $M$ such that
$Y\setminus M$ is a graph.  Let $A_1,A_2,\ldots,A_t$ be the components of $Y\setminus M$.
Perform a contraction on $Y$ by contracting each component $A_i$, $i=1,2,\ldots,t$, to a
single vertex.  Then remove every loop corresponding to the edges of $E(Y)\setminus M$.
The resulting reflexive multigraph $X$ has $Y$ as a generalized truncation.    \end{proof} 

\medskip

There are some facts we may derive from Theorem \ref{nscond} and its proof.  We
state them as separate corollories for clarity and as an algorithm.  Note that a multiple edge appears in
$\mathrm{TR}(X)$ only when there is an edge of $F(M)$ whose end vertices have the same
label, that is, the edge of $F(M)$ arose from a loop in $X$.  Moreover, because we insert
graphs during the assemblage stage, no edge in $\mathrm{TR}(X)$ may have multiplicity 3
or more.  From the theorem we see that the distinct edges of multiplicity 2 in $\mathrm{TR}(X)$ must
not share any vertices.  This proves the following corollary which actually is a reformulation
of Theorem \ref{nscond}.

\begin{cor} A multigraph $Y$ is a generalized truncation of some reflexive multigraph $X$
if and only if $Y$ has no edges of multiplicity bigger than 2, the edges of multiplicity 2 form
a matching, and there is a perfect matching containing all the edges of multiplicity 2.
\end{cor}

Theorem \ref{nscond} informs us when a multigraph is a generalized truncation of a reflexive
multigraph but we now restrict ourselves to multigraphs for the following reason.  If $Y$ is a
generalized truncation of a reflexive multigraph $X$ of size $m$, then $Y$ clearly is a
generalized truncation of the reflexive multigraph with a single vertex and $m$ loops.
This follows because every vertex of the $m$-matching arising in the excision stage has
the same label which enables use to insert any graph on the $2m$ vertices.  Because
of this we now exclude consideration of loops, that is, we consider only generalized truncations
arising from multigraphs and graphs.  Thus, the generalized truncations themselves always
are graphs.

\begin{defn}\label{source}{\em Given a graph $Y$, define the} source of $Y$, {\em denoted
$\mathrm{src}(Y)$, by $\mathrm{src}(Y)=\{X: Y\mbox{ is a generalized truncation of $X$ and $X$ is a multigraph}\}$.}
\end{defn}

\begin{defn}\label{ipm}{\em A perfect matching $M$ in a graph $Y$ is called} isolating
{\em if no edge of $M$ has both end vertices in the same component of $Y\setminus M$.}
\end{defn} 

The proof of one direction of Theorem \ref{nscond} is algorithmic so that we list the steps for
finding a multigraph in $\mathrm{src}(Y)$.

\begin{description}
\item[Step 1.] Find an isolating perfect matching $M$ in $Y$.  If there is none, then $Y$ is not
a generalized truncation of a multigraph,
\item[Step 2.] If there is an isolating perfect matching $M$ in $Y$, let $A_1,A_2,\ldots,A_t$ be
the components of $Y\setminus M$.
\item[Step 3.] Contract each set $A_i$, $i=1,2,\dots,t$, in $Y$ to a single vertex and remove all
the loops formed.  The remaining multigraph $X$ belongs to $\mathrm{src}(Y)$.
\end{description}

It is natural to wonder when $\mathrm{src}(Y)$ contains a graph.  This does impose
an additional restriction on the isolating perfect matching $M$.  Namely, there cannot be two edges
of $M$ whose end vertices are in the same pair of distinct components $A_i$ and $A_j$.
This proves the following corollary.
 
\begin{cor} The graph $Y$ is a generalized truncation of  a graph $X$ if and only if $Y$
contains an isolating perfect matching $M$ such that no edge of $M$ has both end vertices in the
same component of $Y\setminus M$, and there are no two edges of $M$ having their
end vertices in the same pair of distinct components of $Y\setminus M$.
\end{cor}

It is easy to see from the definition that in general a given reflexive multigraph has
many generalized truncations.  The other direction is more interesting and we state a general
problem that is wide open. 

\smallskip

{\bf Research Problem 1}:  What can we say about $\mathrm{src}(Y)$ for various
families of graphs?

\smallskip 

With regard to Research Problem 1, we can determine the graphs $Y$ that have a unique
source.  As a first step we prove the following lemma.

\begin{lemma}\label{srccon}If $Y$ is a graph for which $|\mathrm{src}(Y)|=1$, then $Y$ is
connected. 
\end{lemma}
\begin{proof} Let $Y$ be a graph for which $\mathrm{src}(Y)\neq\emptyset$ and $Y$ is not
connected.   We know that $Y$ has an isolating perfect matching $M$ so that $M$ restricted to
each component $\Gamma$ of $Y$ is an isolating perfect matching for $\Gamma$.  We then
obtain a multigraph $X_{\Gamma}$ which is a source for $\Gamma$.

The disconnected multigraph $X$ formed by the union of the $X_{\Gamma}$s over the components
of $Y$ belongs to $\mathrm{src}(Y)$.  If we now amalgamate two components of $X$ at a
single vertex, then this yields another element of $\mathrm{src}(Y)$ and the result follows.   \end{proof}

\bigskip

Let $\alpha K_n$ denote the complete multigraph for which every edge has multiplicity $\alpha$. 
When $\alpha=1$, simply write $K_n$.  In general, a {\it complete multigraph} is a multigraph
in which every pair of distinct vertices is joined by at least one edge and the multiplicities may
vary over the various edges.

\begin{lemma}\label{cmpltmlt} If $Y$ is a graph that is a generalized truncation with a unique
source $X$, then $X$ is a complete multigraph.
\end{lemma}
\begin{proof}  Let $Y'$ be a graph that is a generalized truncation and let $X'\in\mathrm{src}(Y')$. 
If there are two vertices $u,v\in V(X')$ not joined by an edge, then we may identify $u$ and $v$
to obtain another multigraph of smaller order in $\mathrm{src}(Y')$.  The result now follows.   \end{proof}

\begin{thm}\label{usrc} Let $Y$ be a graph that is a generalized truncation.  If $Y$ has a unique
isolating perfect matching $M$ and there is at least one edge of $M$ joining any two components
of $Y\setminus M$, then $|\mathrm{src}(Y)|=1$. 
\end{thm}
\begin{proof} Each component of $Y\setminus M$ corresponds to the same vertex label on the
ends of the edges of $M$ incident with vertices of the component.  Because there is at least
one edge of $M$ between two distinct components and a source multigraph has no loops, the
labels on the vertices of the different components are distinct.  The result now follows.     \end{proof}

\bigskip

Using Theorem \ref{usrc}, we see that the unique source of the cartesian product of $C_3$ and
$K_2$ is $3K_2$.  On the other hand, the hypothesis that $Y$ has a unique isolating perfect
matching is not necessary because the 4-cycle $C_4$ has two isolating perfect matchings even
though the unique source is $2K_2$.  Finally, the cartesian product of $P_4$, the path of order
4, and $K_2$ has an isolating perfect matching that yields $4K_2$ as a source, and it has an
isolating perfect matching that yields the multipath $2P_3$ as a source.  From this it is seen
that determining the isolating perfect matchings is a key towards progress on the research
problem.

\section{Connectivity} 

A fundamental question is when is a generalized truncation of a multigraph connected?  After
the excision stage of forming a generalized truncation of a multigraph $X$, the vertex-labelled
matching $F(M)$ is obtained.  Form an auxiliary labelled graph $\hat{X}$ by contracting each
edge of $F(M)$ to a vertex and label the vertex with the 2-set consisting of the labels of the
ends of the corresponding edge.  Form the edges of $\hat{X}$ by letting two labelled vertices
be adjacent if and only if their label sets have non-empty intersection. 

Let $Y$ be a generalized truncation of the multigraph $X$.  The {\it projection} of $Y$ into
$\hat{X}$ is the subgraph of $\hat{X}$ obtained by including an edge joining two vertices
with labels $\{x,y\}$ and $\{z,w\}$ if and only if there is an edge of $Y$ joining a vertex of
the edge with labels $x,y$ and a vertex of the edge with labels $z,w$.

\begin{thm}\label{con} The generalized truncation $Y$ of a multigraph $X$ is connected if
and only if the projection of $Y$ into $\hat{X}$ is connected.
\end{thm}
\begin{proof} The trivial proof is left to the reader.    \end{proof}

\bigskip

The auxiliary graph $\hat{X}$ provides an obvious constructive method for producing a
connected generalized truncation of $X$.  Choose a spanning tree $\hat{T}$ of $\hat{X}$. 
For each edge of $\hat{T}$, insert one edge between the corresonding edges of $F(M)$.
It is easy to see that the result is a generalized truncation of $X$ which is itself a tree.

We follow the convention of not specifying the noun ``vertex'' when discussing the
vertex connectivity of a multigraph, whereas, we employ the word ``edge'' when discussing
the edge connectivity.  That is, we shall use the notations $k$-connected and $k$-edge-connected. 
Denote the connectivity and edge connectivity of a multigraph $X$ by
$\kappa(X)$ and $\kappa'(X)$, respectively.

The following material on connectivity applies the results of the classical theorems by
Menger that tell us that the minimum number of vertices that must be deleted from a
multigraph $X$ in order to separate two vertices $u,v\in V(X)$ equals the maximum
number of internally disjoint paths in $X$ whose terminal vertices are $u$ and $v$.
The edge analogue replaces ``number of vertices'' with ``number of edges,'' and
``internally disjoint'' with ``mutually edge-disjoint.''  Thus, a multigraph $X$ is $k$-connected
if and only if every pair of distinct vertices is joined by $k$ internally disjoint paths,
and is $k$-edge-connected if and only if every pair of distinct vertices is joined by $k$ mutually
edge-disjoint paths.  That is why the following proofs talk about paths joining vertices. 

\begin{thm}\label{econ} If $Y$ is a generalized truncation of a multigraph $X$, then
$\kappa'(Y)\leq\kappa'(X)$.
\end{thm}
\begin{proof} It is clear that if $X$ is disconnected, then every
generalized truncation is disconnected.  So assume $X$ is connected and consider a
minimum edge cut $\mathcal{E}$.  The multigraph $X\setminus\mathcal{E}$ has two
components.  Let $A$ be the vertices of one component and $B$ be the vertices of the
other component.  It is clear that the only edges of any generalized truncation $Y$ of
$X$ which may have a label from $A$ and a label from $B$ are the edges of the
matching in $F(M)$ arising from $\mathcal{E}$.  Thus, these edges separate $Y$ into at least two
components.  The result now follows.      \end{proof}

\bigskip

The interesting problem that now arises is how we guarantee that a multigraph $X$ and a
generalized truncation of $X$ have the same edge connectivity.   The next lemma is useful
for subsequent results but first we have a definition followed by a discussion of a method
to be employed frequently.

\begin{defn}\label{cohes}{\em A generalized truncation is said to be} cohesive {\em when
every consituent is connected.}
\end{defn}

Given a path or a cycle in a graph $X$, we now discuss how to expand it to a path or cycle
in $Y=\mathrm{TR}(X)$.  Let $uvw$ be three successive vertices in a path $P$ in $X$.  The edges
$uv$ and $vw$ are in $F(M)$ and the two occurrences of $v$ give rise to two distinct vertices
$v(x)$ and $v(y)$ in $\mathrm{con}(v)$ which are the ends of the edges labelled with $v$.
If there is a path in $\mathrm{con}(v)$ from $v(x)$ to $v(y)$, then we can add this path to
the edges of $F(M)$ that arise from $P$.  If we are able to do this for each constituent, we
obtain a path in $Y$ based on $P$.  We call this {\it an expansion of} $P$ to $Y$.  It is
obvious what we mean by an expansion of a cycle. 

\begin{lemma}\label{useful} Let $Y$ be a cohesive generalized truncation of a multigraph $X$.
If $\mathcal{E}$ is an edge cut of $Y$ using only edges from $F(M)$, then the edges of $X$
corresponding to the edges of $\mathcal{E}$ form an edge cut of $X$.
\end{lemma}
\begin{proof} Let $Y$ and $\mathcal{E}$ be as hypothesised.  There is an edge of
$\mathcal{E}$ whose end vertices $x$ and $y$ are in different components of
$Y\setminus\mathcal{E}$ because $\mathcal{E}$ is an edge cut.  Also, $x$ and $y$ belong
to different constituents because the edges of $\mathcal{E}$ belong to $F(M)$.  Let
$x\in\mathrm{con}(u)$ and $y\in\mathrm{con}(v)$, respectively.

Let $\mathcal{E}'$ be the edges in $X$ corresponding to the edges of $\mathcal{E}$.
Assume that $\mathcal{E}'$ is not an edge cut of $X$.  Then there is a path $P$ in $X\setminus
\mathcal{E}'$ whose end vertices are $u$ and $v$.  The edges of $P$ belong to $M$ and
$E(X)\setminus\mathcal{E}'$.  Hence, these edges form a matching in $Y\setminus\mathcal{E}$
and successive edges share a label, say $w$.  There is a path in $\mathrm{con}(w)$ joining
the two vertices with the same label. This yields a path in $Y\setminus\mathcal{E}$ joining $x$
and $y$ which is a contradiction.  Therefore, $\mathcal{E}'$ is an edge cut in $X$ as claimed.     \end{proof}

\begin{defn}\label{cmplt}{\em A generalized truncation is called} complete {\em if every
constituent graph is complete.}
\end{defn} 

\begin{thm}\label{nkreg} If $X$ is a $k$-edge-connected multigraph, $k\geq 2$, then a
complete generalized truncation $Y$ of $X$ is $k$-edge-connected.
\end{thm} 
 \begin{proof}  Let $Y$ be a generalized truncation of $X$.  Note that every constituent
graph has order at least $k$ because $X$ is $k$-edge-connected.

First choose two vertices $x$ and $y$ in the same constituent $\mathrm{con}(v)$.  If
$|\mathrm{con}(v)|>k$, then it is trivially the case that there are $k$ mutually edge-disjoint
paths whose terminal vertices are $x$ and $y$.  Hence, we assume that $|\mathrm{con}(v)|
=k$.  

Because the latter
subgraph is complete, we may choose the edge $xy$ and the 2-paths $xzy$, as $z$ runs
through the remaining vertices of $\mathrm{con}(v)$, to obtain $k-1$ mutually edge-disjoint
paths whose terminal vertices are $x$ and $y$.  If we find an additional path from $x$ to $y$ that is
edge-disjoint from the other paths, then we shall have shown that an edge-separating set for
the vertices $x$ and $y$ has cardinality at least $k$.

Let $uv$ and $wv$ be the two edges of $X$ giving rise to the vertices $x$ and $y$ in
$\mathrm{con}(v)$.  There are two edge-disjoint paths joining $u$ and $w$ in $X$ because
$X$ is 2-edge-connected.  If one of the paths does not contain $v$, then there is a cycle
containing the 2-path $uvw$.  Expansion of this cycle produces a path from $x$ to $y$
in $Y$ that uses none of the edges of the initial $k-1$ paths.  

On the other hand, if both paths in $X$ contain $v$, then the union of the two paths is an
eulerian subgraph $X'$ of $X$ in which $\mathrm{val}(v)=4$.  If the edge $uv$ does not
belong to $X'$, then choose one of the paths from $u$ to $w$ and replace the subpath
from $u$ to $v$ with the edge $uv$.  This results in a smaller eulerian subgraph in
which $\mathrm{val}(v)=4$.  We may repeat this operation for the edge $wv$ so that
we may assume that both $uv$ and $wv$ belong to $X'$.  Thus, there is an Euler tour
that starts with the edge $vu$ from $v$ to $u$ and finishes with the edge $wv$.  We
may then use expansion on this Euler tour and obtain a path from $x$ to $y$ in $Y$
that does not use any edge of the first $k-1$ paths.  To see this we need to consider
vertices of valency 4 in $X'$.

If $z\neq v$ has valency 4 in $X'$, then the four edges incident with $z$ correspond to
four distinct vertices of $\mathrm{con}(z)$ and it is easy to see that extension may be
achieved by two edges having no vertices in common in $\mathrm{con}(z)$.
On the other hand, the Euler tour passes through $v$ once in the interior of the tour.
This corresponds to two vertices of $\mathrm{con}(v)$ distinct from both $x$ and $y$.
The edge joining them is not used in any of the first $k-1$ paths and this gives us an
extension of the Euler tour that is another path from $x$ to $y$ in $Y$.

When $x$ and $y$ lie in different constituents $\mathrm{con}(u)$ and $\mathrm{con}(v)$,
respectively, the existence of $k$ edge-disjoint paths joining them in $Y$ is easy to establish.   
There are $k$ edge-disjoint paths in $X$ whose terminal vertices are $u$ and $v$.  Use
expansion to obtain $k$ edge-disjoint paths in $Y$ from $\mathrm{con}(u)$ to
$\mathrm{con}(v)$.  We then may use edges in each of the constituents to make the terminal
vertices of each path $x$ and $y$ because the constituents are complete graphs. 
This completes the proof.    \end{proof}

\begin{cor}\label{kreg} If $X$ is a $k$-regular, $k$-edge-connected multigraph, $k\geq 2$, then
a generalized truncation $Y$ of $X$ is $k$-edge-connected if and only if it is complete.
\end{cor}
\begin{proof} If every constituent is complete, then $Y$ is $k$-edge-connected by Theorem
\ref{nkreg}.   If there is a constituent graph $\mathrm{con}(v)$ which is not complete, then
there are two vertices $x$ and $y$ of $\mathrm{con}(v)$ which are not adjacent.  This implies
that $\mathrm{val}(x)<k$ in $Y$.  This, in turn, implies $\kappa'(Y)\leq k-1$.  By the
contrapositive, if $Y$ is $k$-edge-connected, then each constituent graph is complete.   \end{proof}

\begin{thm}\label{vconn} If $X$ is a k-connected multigraph, $k\geq 2$, then a complete
generalized truncation $Y$ of $X$ is k-connected.
\end{thm}
\begin{proof} Let $Y$ be a complete generalized truncation of the $k$-connected multigraph $X$.  
First consider two vertices $x$ and $y$ of $Y$
which belong to constituents $\mathrm{con}(u)$ and $\mathrm{con}(v)$, $u\neq v$, respectively.
There are $k$ internally disjoint paths from $u$ to $v$ in $X$.   Expanding the paths gives
us $k$ mutually vertex-disjoint paths from vertices of $\mathrm{con}(u)$ to vertices of
$\mathrm{con}(v)$ in $Y$.  We then use edges of each of the constituents to obtain $k$ internally
disjoint paths from $x$ to $y$ and may do so because each constituent is complete.

Suppose now that $x$ and $y$ belong to the same constituent $\mathrm{con}(v)$.  If
$\mathrm{val}(v)>k$ in $X$, then there are trivially at least $k$ internally disjoint paths joining
$x$ and $y$ in $\mathrm{con}(v)$ because it is a complete graph.  So we may assume that
$|\mathrm{con}(v)|=k$.

There are $k-1$ internally disjoint paths joining $x$ and $y$ in $\mathrm{con}(v)$ and we
need to find one more path that is internally disjoint from the $k-1$ paths.  Let $u'x$ and
$w'y$ be the edges of $Y$ incident with $x$ and $y$ such that $u'\in\mathrm{con}(u)$
and $w'\in\mathrm{con}(w)$, where $u,v\mbox{ and }w$ are distinct.  Because $X$ is
2-connected, there is a path in $X$ missing the vertex $v$.  Extending this path gives a
path $Q$ in $Y$ from a vertex of $\mathrm{con}(u)$ to a vertex of $\mathrm{con}(w)$ not
containing any vertex of $\mathrm{con}(v)$.  We may then use edges in the two
constituents, if necessary, to obtain a path in $Y$ from $u'$ to $w'$.  Then adding the edges
$u'x$ and $w'y$ gives the desired path in $Y$ completing the proof.     \end{proof}

\medskip

The proof of the following corollary is easy and shall not be given.

\begin{cor} If $X$ is a $k$-connected $k$-regular graph, then a generalized truncation of
$X$ is $k$-connected if and only if every constituent graph is complete.
\end{cor} 

{\bf Research Problem 2}.  Determine conditions on the original multigraph $X$ and the
constituent graphs of a generalized truncation $\mathrm{TR}(X)$ that determine the
connectivity and/or the edge-connectivity of $\mathrm{TR}(X)$.

\section{Eulerian Truncations}

Recall that an {\it Euler tour} in a multigraph is a closed trail that covers each edge precisely
once.  A multigraph $X$ is {\it eulerian} if it possesses an Euler tour.  Also recall the following
well-known theorem of Euler.

\begin{thm} A connected multigraph $X$ is eulerian if and only if every vertex has even
valency.
\end{thm}

As we shall soon see, determining when a generalized truncation $Y$ is eulerian is straightforward.

\begin{thm}\label{euler} Let $X$ be a connected multigraph.  Every component of a generalized
truncation $Y$ of $X$ is eulerian if and only if $X$ is eulerian and every constituent has only
vertices of odd valency.
\end{thm}
\begin{proof} Let $X$ be an eulerian multigraph so that every vertex of $X$ has even valency.
This implies that every constituent $\mathrm{con}(u)$ has even order.  The valency of a
vertex $x\in\mathrm{con}(u)$ in $Y$ is one plus its valency in $\mathrm{con}(u)$.  Thus, if
every vertex in every constituent has odd valency in the constituent, then every vertex has
even valency in $Y$.  Hence, every component of $Y$ is eulerian.

On the other hand, if every component of $Y$ is eulerian, then every vertex has even valency
in $Y$.  This implies that evey vertex has odd valency in its constituent.  This, in turn, implies
that every component has even order.  Then every constituent has even order which implies
that $X$ is eulerian because it is given that $X$ is connected.   \end{proof} 

\medskip

The preceding is a local theorem because we need only consider each constituent in order to
achieve the concusion.  However, considering the constituents individually does not guarentee
that the generalized truncation $Y$ itself is eulerian because it may not be connected.  So we
need to consider the structure of $X$ in order to obtain a generalized truncation $Y$ that is
eulerian. 

\section{Hamiltonicity}

There are three hamiltonicity problems we consider in this section.  The first deals with the
hamiltonian problem, that is, does a graph contain a Hamilton cycle.  The hamiltonian
problem is one of the earliest problems arising in graph theory, and is one that has been
widely studied in many contexts.  

It is apparent that  that there will be no easy answers regarding the existence of Hamilton
cycles in generalized truncations.  We may safely say this because there are many ways
for a Hamilton cycle to contain all the vertices of a constituent graph.  For example, it might
enter a constituent once and pass through all the vertices of the constituent before exiting.
On the other hand, it might enter and exit multiple times.  What is the case is that a Hamilton
cycle in a generalized truncation partitions a constituent into a collection of vertex-disjoint
paths covering the vertices of the constituent. 

The following theorem appears in \cite{A1}.

\begin{thm}\label{ham} If $\mathrm{TR}(X)$ is a complete generalized truncation of a connected
multigraph $X$, then $\mathrm{TR}(X)$ is hamiltonian if and only if $X$ contains a spanning
eulerian subgraph.
\end{thm} 

The preceding theorem is special because each of the constituent graphs is complete suggesting
the following question. 

\medskip

{\bf Research Problem 3}.  Determine conditions on the source multigraph and constituents
that imply a generalized truncation is hamiltonian.

\medskip

The second hamiltonicity problem we consider is Hamilton connectivity.  Recall that a multigraph
$X$ is {\it Hamilton-connected} if for every pair of vertices $u$ and $v$ in $X$ there is a
Hamilton path in $X$ whose terminal vertices are $u$ and $v$.  Similarly, a bipartite multigraph
$X$ with parts of the same cardinality is {\it Hamilton-laceable} if for any two vertices in
opposite parts there is Hamilton path in $X$ from one to the other. 

The only Hamilton-connected graph with a vertex of valency 1 is $K_2$.  The generalized
truncation of $K_2$ is $K_2$ itself so that every generalized truncation of a Hamilton-connected
graph with a vertex of valency 1 is Hamilton-connected.

The Hamilton-connected multigraphs with a vertex of valency 2 are $K_3$ and $2K_2$ (an
edge of multiplicity 2).  The complete generalized truncations of these two multigraphs are
$C_6$ (the cycle of length 6) and $C_4$.  Neither of them are Hamilton-connected so that
there are no Hamilton-connected generalized truncations of either $K_3$ or $2K_2$.

From the preceding comments we may assume that the multigraphs under consideration
have minimum valency at least 3.  Note that once a multigraph has a vertex of valency 3 or
more, then its complete generalized truncation is not bipartite.  Hence, the complete generalized
truncation of a bipartite graph may not be bipartite.  However, bipartiteness may not be
a barrier to generalized truncations being Hamilton-connected.  For example, the complete
bipartite $K_{3,3}$ is easily seen to be Hamilton-laceable.  It turns out that its complete
generalized truncation is Hamilton-connected.

\begin{thm}\label{hc} A generalized truncation of the complete graph $K_n$, $n>3$, is
Hamilton-connected if every constituent graph is Hamilton-connected.
\end{thm}
\begin{proof} Let $Y$ be a generalized truncation of $K_n$, $n>3$, in which every constituent
graph is Hamilton-connected.  Let $x$ and $y$ be vertices of $Y$ in different constituent graphs
$\mathrm{con}(u)$ and $\mathrm{con}(v)$, respectively.  Let $[u,w]$ be the edge of $K_n$
such that $x$ is the vertex of $\mathrm{con}(u)$ corresponding to $u$ and, similarly, let
$[z,v]$ be the edge of $K_n$ such that $y$ is the vertex of $\mathrm{con}(v)$
corresponding to $v$.  We can find a Hamilton path $P$ in $K_n$ from $u$ to $v$ such that
$w$ is not the vertex following $u$ on $P$, and $z$ is not the vertex preceding $v$ on $P$
because $n\geq 4$.

It is now easy to find a Hamilton path from $x$ to $y$ by extending $P$.  The vertex of
$\mathrm{con}(u)$ corresponding to the vertex $u$ in $P$ is $x'\neq x$.  Because
$\mathrm{con}(u)$ is Hamilton-connected, there is a path from $x$ to $x'$ spanning all
the vertices of $\mathrm{con}(u)$.  It is easy to use all of the vertices of the constituent
graphs as we work along $P$ because they are Hamilton-connected and the entering and
departing vertices are distinct.  The completion of the Hamilton path from $x$ to $y$ in
$Y$ in $\mathrm{con}(v)$ is done in the same way as the path was started in $\mathrm{con}(u)$.  

Now let $x$ and $y$ both belong to $\mathrm{con}(u)$.  A little more care
needs to be taken in this case.  Because $n\geq 4$, $|\mathrm{con}(u)|\geq 3$.
Let $[u,v],[u,w]\mbox{ and }[u,z]$ be edges of $K_n$ corresponding to the vertices
$x,y\mbox{ and }z$ in $\mathrm{con}(u)$, where $z$ will be specified shortly.

Consider a path $Q$ in $\mathrm{con}(u)$ from $x$ to $y$ spanning all the vertices
of this constituent graph.  Let $z$ be the vertex preceding $y$ on $Q$.  Choose a
Hamilton path $P$ in $X$ from $u$ to $w$ starting with the edge $[u,z]$.  We extend
$P$ in the following way.  Start by removing the last vertex of $Q$ so that we have a
path from $x$ to $z$ using all the vertices of $\mathrm{con}(u)$ other than $y$.
Now extend $P$ through the other constituent graphs as before until reaching
$\mathrm{con}(w)$.  Extend the path in the latter constituent graph so that it ends
at the vertex corresponding to the edge $[u,w]$ in $X$.  Then add the edge to
$y$ and we have the desired Hamilton path in $Y$.     \end{proof} 

\medskip

The conditions for Theorem \ref{hc} are special and suggest two further problems.

\medskip

{\bf Research Problem 4}.  If $X$ is a Hamilton-connected or Hamilton-laceable multigraph
with minimum valency at least 3, is the complete generalized truncation of $X$ Hamilton-connected?

\medskip

{\bf Research Problem 5}.  What conditions on the source multigraph $X$ and the
constituents of a generlized truncation $Y$ of $X$ guarantee that $Y$ is Hamilton-connected?

\medskip

The final problem we consider deals with Hamilton decompositions.  A regular graph is
{\it Hamilton-decomposable} if its edge set can partitioned into Hamilton cycles when the
valency is even, and into Hamilton cycles and a single perfect matching when the valency is
odd.  For the next result we require two facts that we encapsulate as a lemma.  These facts
are based on the Walecki decompositions given in \cite{A2}.  The first fact is presented
directly in \cite{A2}.  The second fact is obtained by removing the diameter edge
from each Hamilton cycle in the decomposition of a complete graph of odd order into
Hamilton cycles which also is given in \cite{A2}.

\begin{lemma}\label{help} Let $X$ be a complete graph of order $n$.

{\rm (i)} If $n$ is even, then $X$ has a decomposition into $n/2$ Hamilton paths.

{\rm (ii)} If $n$ is odd, then $X$ has a decomposition into $(n-1)/2$ Hamilton paths and a matching
with $(n-1)/2$ edges.
\end{lemma}

\begin{thm}\label{hdecomp} If $X$ is a Hamilton-decomposable graph, then the complete
generalized truncation of $X$ also is Hamilton-decomposable.
\end{thm}
\begin{proof} Let $X$ have a decomposition into Hamilton cycles $H_1,H_2,\ldots,H_n$.
Let $Y$ denote the complete generalized truncation of $X$.  Note that each constituent of
$Y$ has order $2n$.  The edges of $H_i$ in $Y$ intersect each constituent in two vertices.
Hence, the $2n$ vertices of a given constituent are partitioned into $n$ 2-sets. By Lemma
\ref{help}(i), we may decompose a constituent into $n$ spanning paths such that the end
vertices of each path belong to the same 2-set.  It now is obvious that we may expand each
Hamilton cycle $H_i$ of $X$ into a Hamilton cycle in $Y$ using the spanning paths of the
constituents.  This decomposes $Y$ into $n$ Hamilton cycles.

When $X$ has a decomposition into $n$ Hamilton cycles and a single perfect matching, we
slightly modify the preceding construction.  Each constituent now has odd order so we use
Lemma \ref{help}(ii) to decompose the constituent into $n$ spanning paths and an $n$-matching.
The $n$-matching misses precisely one vertex of the constituent and we make sure that
missing vertex is the vertex which is incident with the matching edge of $X$ that is incident
with a vertex of the constituent.  It is now easy to see how to complete the Hamilton
decomposition of $Y$.     \end{proof}

\medskip

There are other ways to obtain a Hamilton decomposition of a generalized truncation.
For example, if we start with a spanning eulerian subgraph of valency 4 in $X$, we may
use that to obtain a Hamilton cycle in $Y$.  This suggests the following problem.

\medskip

{\bf Research Problem 6}.  Find conditions on the source graph and the constituents that
produce a Hamilton-decomposable generalized truncation.

\section{Planarity}

Planarity is another basic topic that has been studied extensively in graph theory.  It is natural
to consider which generalized truncations are planar.  After the excision stage before any edges
have been added to the constituents, the generalized truncation certainly is planar which
suggests two questions.  First, what can we say about planarity in terms of the number of edges
we introduce in the constituents.  Second, what can we say about planarity if we insist that
the generalized truncation is cohesive.  We now investigate the second question.

\begin{lemma}\label{pl1} If $X$ is a non-planar graph, then every cohesive generalized
truncation of $X$ is non-planar.
\end{lemma}
\begin{proof} Because $X$ is non-planar, it has either $K_{3,3}$ or $K_5$ as a minor.
If $Y$ is a cohesive generalized truncation of $X$, then $X$ is a minor of $Y$ by contracting
each constituent of $Y$ to a single vertex and removing the loops.  Thus, $Y$ has either a
$K_{3,3}$-minor or a $K_5$-minor as the minor relation is transitive.     \end{proof}

\medskip

Because of Lemma \ref{pl1}, we now consider cohesive generalized truncations of planar
graphs and describe a process that produces a planar cohesive generalized truncation.  Let
$X$ be a plane graph, that is, it is given embedded in the plane with no edges crossing.
Draw a small closed disc around each vertex of $X$ so that none of the discs overlap.  Remove the
intersection of each edge with the interior of the discs surrounding its end vertices, and let the
intersections of the edges with the boundaries of the discs be the end vertices of the fragments
of the original edges.

After performing the preceding operations, we have the perfect matching $F(M)$ embedded in
the plane.  Recall that a graph is {\it outerplanar} if it has an embedding in the plane so that
every vertex belongs to the boundary of the infinite face.  If we now insert an outerplanar graph
for each constituent, it is clear that the
resulting generalized truncation is planar.  However, we shall now see that there are planar
generalized truncations for which there are constituents that are not outerplanar.  To get a
handle on this we use the following result from \cite{C1}.

\begin{thm}\label{outer} A graph is outerplanar if and only it it contains no subgraph homeomorphic
to $K_{2,3}$ or $K_4$.
\end{thm}

\begin{picture}(300,240)(-20,-50)
\put(0,0){\circle*{5}}
\put(300,0){\circle*{5}}
\put(150,75){\circle*{5}}
\put(150,150){\circle*{5}}
\put(0,0){\line(1,0){300}}
\put(0,0){\line(2,1){150}}
\put(150,75){\line(2,-1){150}}
\put(0,0){\line(1,1){150}}
\put(150,75){\line(0,1){75}}
\put(150,150){\line(1,-1){150}}
\put(-10,0){1}
\put(148,155){2}
\put(305,0){3}
\put(140,75){4}
\multiput(140,30)(20,0){2}{\line(0,1){20}}
\multiput(140,30)(0,20){2}{\line(1,0){20}}
\put(148,35){$A$}
\put(150,75){\line(0,-1){25}}
\qbezier(0,0)(70,20)(140,40)
\put(125,-20){\sc Figure 1}
\end{picture}

Consider Figure 1.  Suppose that the vertices labelled 1 through 4 are the vertices of $\mathrm{con}(u)$
for a vertex $u$ of valency 4 in a planar graph $X$ and the graph depicted in the figure is a subgraph
of a generalized truncation of $X$.  These four vertices have been joined
to form a constituent that is $K_4$.  The crucial vertex here is 4 because the corresponding edge incident with $u$
in $X$ cannot pass through the edges of the 3-cycle formed by 1, 2 and 3.  Hence, this edge must
be the edge from 4 to the subgraph indicated by $A$.  The edges of $X$ incident with $u$ corresponding to 1 and 3
may or may not be incident with vertices in $A$.  Figure 1 has been drawn so that the edge corresponding
to 1 also joins a vertex in $A$.  This figure indicates how a generalized truncation of a planar graph
may possess a constituent graph which is not outerplanar.

Let $Y$ be a planar cohesive generalized truncation of a planar graph $X$.
If we have a connected constituent that contains a subdivision of either $K_4$ or $K_{2,3}$, then there
is some vertex $x$ which is contained in the interior of the region bounded by the constituent.  Every vertex of the
constituent is labelled with $u$ so that the edge of $X$ incident with $u$ corresponding to $x$ has its other end
in some face $F$ of the constituent. Hence, there is a subgraph $Y'$ of $Y$ contained in the face $F$.  Any edges
coming into $Y'$ from vertices on the boundary of $F$ have label $u$ on the vertex from the boundary.
Thus, $u$ is a cut vertex in $X$ for whichever edges of $Y'$ initially belonged to $X$.  This implies that $X$
is not 2-connected.   The next theorem now follows.

\begin{thm} A cohesive generalized truncation $Y$ of a 2-connected planar graph $X$ is planar if and only if
every constituent of $Y$ is outerplanar.
\end{thm}

\section{Colorings}

We now consider vertex and edge colorings of generalized truncations. Recall that a {\it proper coloring} of a
multigraph $X$ is a coloring of the vertices so that adjecent vertices do not have the same color.  Similarly, a
{\it proper edge coloring} is a coloring of the edges so that adjacent edges do not have the same color.  The
{\it chromatic number} of $X$, denoted $\chi(X)$, is the fewest number of colors for which a proper coloring
exists, and the {\it chromatic index}, denoted $\chi'(X)$, is the fewest number of colors for which a proper
edge coloring exists. 

A small hint of the kind of behavior that may occur is exemplified by the following.  The graph $K_3$ has both
chromatic number and chromatic index 3.  The complete generalized truncation is the 6-cycle which has
chromatic number and chromatic index 2.  On the other hand, if $X$ is a bipartite graph, then we need at
least $k$ colors to color the vertices of a truncation, where $X$ has a vertex of valency $k$.  So we may
need to introduce many colors when we move from a graph with chromatic number 2 to a generalized truncation.

Vizing's well-known theorem tells us that the chromatic index of a graph equals the maximum
valency or the maximum valency plus one.  This, in turn, leads to a classification of graphs as
follows.  A graphs is {\it class} I if its chromatic index is equal to its maximum valency and is
{\it class} II otherwise.

\begin{thm}\label{color} If $X$ is a class {\em I} graph, then its complete generalized truncation also is
class {\em I}.  If $X$ is a class {\em II} graph and its maximum valency is even, then its
complete generalized truncation is class {\em I}.
\end{thm}
\begin{proof} Let $X$ be a graph whose maximum valency $d$ is even and let its complete
generalized truncation be $Y$.  Every constituent of $Y$ of order $d$ admits a proper edge
coloring with $d-1$ colors because $d$ is even.  Any constituent of order less than $d$ admits
a proper edge coloring with at most $d-1$ colors.  The edges of $Y$ with end vertices in
different constituents form a perfect matching in $Y$.  Color all of these edges with a single
new color, thereby obtaining a proper edge coloring of $Y$ with $d$ colors.  The maximum
valency of $Y$ is $d$ so that $Y$ is class I.

Now suppose that $X$ has maximum valency $d$ and $d$ is odd.  We know that $X$ is class I
by hypothesis. Because $X$ is class I, it has a proper edge coloring using $d$ colors.  In forming
$Y$, retain the colors on the edges between the constituents.  We now describe decompositions
of the constituents into matchings which may be used to color the edges of the constituents so
that we obtain a proper edge coloring of $Y$ without introducing any new colors, thereby
establishing that $Y$ is class I.  

If a constituent has odd order $n$, then decompose its edges into $n$ matchings of size
$(n-1)/2$ so that for each vertex of the constituent there is a unique $(n-1)/2$-matching
missing the given vertex. 

If a constituent has even order $n$, we do something unusual.  Add a new artificial vertex
giving us a complete graph of odd order $n+1$.  We then take the decomposition of
$K_{n+1}$ into $n+1$ matchings of size $n/2$.  Now remove the artificial vertex leaving
us with a decomposition of $K_n$ into one perfect matching and $n$ matchings of size
$(n-2)/2$, where each of the latter matchings miss precisely two vertices.

We now prescribe how to color the edges of the constituents.  If a constituent $\mathrm{con}(u)$
has odd order $n$, then $u$ is incident with $n$ edges with distinct colors in $X$.  If a
given edge incident with $u$ in $X$ has color $\alpha$, then the corresponding edge in $Y$ still
has color $\alpha$.  There is a unique matching in the decomposition of $\mathrm{con}(u)$
missing the vertex of $\mathrm{con}(u)$ incident with the edge of color $\alpha$.  Color
the edges of this matching with color $\alpha$. 

If a constituent $\mathrm{con}(u)$ has even order $n$, then $n<d$ and there is a color
$\alpha$ such that no edge of color $\alpha$ is incident with $u$ in $X$.  Then use $\alpha$
to color the edges of the perfect matching in $\mathrm{con}(u)$.  Color the remaining edges
of $\mathrm{con}(u)$ as done in the case when $n$ is odd.  We now have a proper edge
coloring of $Y$ with $d$ colors so that $Y$ is class I.     \end{proof}

\medskip

There is a notable missing possibility in Theorem \ref{color}, namely, $X$ is class II and its
maximum valency is odd.  As is typical for a situation such as this, we look at the Petersen graph.
It is not difficult to see that the complete generalized truncation $Y$ of the Petersen graph is
class II.  Suppose this was not the case.  Then $Y$ would have a 1-factorization and the
union of two of the 1-factors would form a 2-factor of $Y$ whose components would be
cycles of even length.  Because a 2-factor must contain every vertex of $Y$, each cycle
of the 2-factor must use all three vertices of a constituent when it passes through a
constituent.  Thus, the 2-factor of $Y$ corresponds to a 2-factor of the Petersen graph.
However, all the 2-factors in the Petersen graph consist of two 5-cycles which implies the
only 2-factors in $Y$ consist of two 15-cycles.  A 15-cycle cannot be a cycle in the union of
two 1-factors.

\medskip

{\bf Research Problem 7}.  Characterize the class II generalized truncations of multigraphs.

\medskip

\begin{cor} Let $X$ be a regular graph of valency $d$.  If $d$ is even or $X$ is class I, then
the complete generalized truncation $Y$ of $X$ admits a 1-factorization.
\end{cor}
\begin{proof} The complete generalized truncation of $X$ is regular of valency $d$ and is
class I by Theorem \ref{color}.  This implies that each color class of edges must be a 1-factor.
The result now follows.    \end{proof}

\medskip

{\bf Research Problem 8}.  Determine conditions on the source multigraph and constituents
so that a generalized truncation has a 1-factorization.

\medskip

The spectrum problem for chromatic indices of generalized truncations of a given graph $X$
is straightforward because the minimum generalized truncation is a perfect matching of size
$|E(X)|$ for which the chromatic index is 1.  The maximum value occurs for the chromatic index
of the complete generalized truncation of $X$ which is either the maximum valency of $X$
or the maximum valency plus one.  By adding one edge at a time and realizing the chromatic
index stays the same or increases by one, it is easy to that there are generalized truncations
of $X$ realizing all possible values between one and the upper bound.  However, the problem
takes on more interest if we restrict ourselves to cohesive generalized truncations.

Given a graph $X$, what is the minimum chromatic index of a cohesive generalized truncation
of $X$?  Once that is known, it is easy to see that all values from that point to the maximum
possible value are achieved by a cohesive generalized truncation.

\begin{thm} Let $X$ be a multigraph with maximum valency $d>2$.  If the chromatic index of
the complete generalized truncation of $X$ is $D$, then for every $k$ satisfying $3\leq k\leq D$,
there is a cohesive generalized truncation of $X$ with chromatic index $k$.
\end{thm}
\begin{proof} The idea is to make the generalized truncation cohesive using as few edges as
possible.  The way to do this is to insert a spanning path on the vertices with the same label
in $F(M)$ so that each constituent is a path.  We then color the edges of the constituents
with one or two colors and note that two colors are required because one of the constituents
has order at least three.  We then color the edges whose ends lie in different constituents
with a third color giving us a cohesive generalized truncation with chromatic index 3.

We then add one edge at a time until reaching the complete generalized truncation.  It is
clear that we achieve a cohesive generalized truncation with chromatic index $k$ for all $k$
satisfying $3\leq k\leq D$.      \end{proof}

\medskip

Recall that Brooks' Theorem \cite{B1} states that the chromatic number of a graph $X$ is
bounded above by its maximum valency unless $X$ is complete or an odd length cycle.  This
gives us a quick proof of the next result.

\begin{thm}\label{chr} If $X$ is a multigraph with maximum valency $d>1$, then its complete
generalized truncation $Y$ satisfies $\chi(Y)=d$.
\end{thm}
\begin{proof} Let $X$ be a multigraph satisfying the hypotheses and let $Y$ be its complete
generalized truncation.  Then $Y$ contains a clique of order $d$ from which it follows that
$\chi(Y)\geq d$.  The result follows from Brooks' Theorem if we show that $Y$ is neither an
odd length cycle nor a complete graph.  The order of $Y$ is even so that it cannot be an odd
length cycle.  The order of $X$ is at least two so that $Y$ contains at least two constituents
and there is at least one constituent $\mathrm{con}(u)$ of order bigger than one.  The edges
between constituents form a perfect matching so that $Y$ is not complete.     \end{proof}

\medskip

Consider the spectrum problem for the chromatic numbers of generalized truncations for a
fixed graph $X$.  Theorem \ref{chr} provides an upper bound so that we want to determine
the minimum chromatic number for a cohesive generalized truncation of $X$.  If we again use
a spanning path for each constituent, then the maximum valency for $Y$ is three except
for a few exceptions.  So Brooks' Theorem tells us the chromatic number for such a generalized
truncation is 3 for the unexceptional graphs.

The exceptions arise  if the maximum valency of $X$ is 2.  The complete generalized truncation
of $2K_2$ is a 4-cycle, the complete generalized truncation of an $n$-cycle is a $2n$-cycle,
and the complete generalized truncation of a path of length $n$ is a path of length $2n$.
All of these graphs have chromatic number 2.  The following result follows from these comments
and Theorem \ref{chr}.

\begin{thm} Let $X$ be a multigraph with maximum valency $d>1$.  If $d=2$, then every
cohesive generalized truncation of $X$ also has chromatic number 2.  If $d>2$, then for
every $k$ satisfying $3\leq k\leq d$, there is a cohesive generalized truncation of $X$
with chromatic number $k$.
\end{thm}

\section{Conclusion}

The topic of generalized truncations of reflexive multigraphs may be viewed as very old in the
sense that a special version of it was studied by the ancient Greeks.  A more general version has
been introduced, studied somewhat and even then only in special circumstances.  The general
version presented in this paper is a further extension in what we see as a natural way to proceed.  

The purpose of this paper is to encourage people to study the many possible directions the
topic may proceed.  We have only scratched the surface.  If others pursue this topic, the
authors will be pleased.

Finally, many of the results presented in this paper are contained in the honours thesis
submitted by the second author in June 2020 to the University of Newcastle.

\end{document}